\documentclass{amsart}

\usepackage{color}
\usepackage{bm}
\usepackage{mathtools}
\usepackage{amsfonts}
\usepackage{amssymb}
\usepackage{graphicx}
\usepackage{epstopdf}
\usepackage{bbold}
\usepackage{amsopn}
\usepackage{multirow}
\usepackage{array}
\usepackage{algorithm}
\usepackage{algpseudocode}
\usepackage{cite}
\usepackage{url}
\usepackage{hyperref}

\algrenewcommand{\algorithmicrequire}{\textbf{Input:}}
\algrenewcommand{\algorithmicensure}{\textbf{Output:}}

\theoremstyle{plain}
\newtheorem{theorem}{Theorem}[section]
\newtheorem{lemma}[theorem]{Lemma}

\theoremstyle{definition}
\newtheorem{remark}[theorem]{Remark}

\newcommand{\hL}{\mathcal{L}}
\newcommand{\dd}{\mathrm{d}}
\newcommand{\abs}[1]{\left\vert #1 \right\vert}
\newcommand{\norm}[1]{\left\| #1 \right\|}
\newcommand{\vertiii}[1]{\left|\!\left|\!\left| #1 \right|\!\right|\!\right|} 
\newcommand{\e}{\mathrm{e}}

\newcommand{\RR}{\mathbb{R}}

\newcommand{\hI}{\mathcal{I}}

\newcommand{\halp}{\widehat{\alpha}}
\newcommand{\hbet}{\widehat{\beta}}
\newcommand{\htaumax}{\widehat{\tau}_{\max}}
\newcommand{\ttaumax}{\widetilde{\tau}_{\max}}
\newcommand{\hxzero}{\widehat{x}_0}

\allowdisplaybreaks[4]

\DeclareGraphicsExtensions{.eps,.pdf,.png,.jpg}

\hypersetup{
  pdftitle={Original energy dissipation of ETDRK methods},
  pdfauthor={C. Quan, P. Zheng, and Z. Zhou}
}

\begin{document}

\title[Original energy dissipation of ETDRK methods]{A unified framework for original energy dissipation of equilibrium-preserving exponential Runge--Kutta methods for $L^2$ gradient flows}

\author{Chaoyu Quan}
\address{School of Science and Engineering, The Chinese University of Hong Kong, Shenzhen, P.~R.~China
(\href{mailto:quanchaoyu@cuhk.edu.cn}{quanchaoyu@cuhk.edu.cn})}
\author{Pinzhong Zheng and Zhi Zhou}
\address{Department of Applied Mathematics, The Hong Kong Polytechnic University, Hong Kong
(\href{mailto:pinzhong.zheng@connect.polyu.hk}{pinzhong.zheng@connect.polyu.hk},
\href{mailto:zhi.zhou@polyu.edu.hk}{zhi.zhou@polyu.edu.hk})}

\begin{abstract}
This paper develops a unified framework for establishing original energy dissipation of equilibrium-preserving exponential time differencing Runge--Kutta (ETDRK) methods for $L^2$ gradient flows. The analysis is based on an auxiliary continuous evolution representation of a general ETDRK method, which allows the energy variation to be controlled without imposing positive-definiteness conditions on the Butcher coefficients or introducing stabilization terms. For a broad class of equilibrium-preserving ETDRK schemes, we derive an explicitly computable time-step restriction. Explicit restrictions are obtained for representative schemes of orders two through five, and the framework also applies to arbitrarily high-order interpolation-based ETDRK methods. For the latter class, the resulting admissible step sizes are substantially larger than those obtained from the previous interpolation-specific analysis. Numerical experiments for the Allen--Cahn equation confirm the theoretical dissipation restrictions.
\end{abstract}



\maketitle

\section{Introduction}\label{sec:intro}

This paper focuses on the construction and analysis of energy-decreasing time discretizations for $L^2$ gradient flows.
Given a free energy functional $E[u]$, the $L^2$ gradient flow is defined as the evolution equation
\begin{equation*}
    u_t = - \frac{\delta E}{\delta u},
\end{equation*}
supplemented with suitable boundary conditions, where $\frac{\delta E}{\delta u}$ is the variational derivative of $E[u]$. A classical reaction--diffusion example, after scaling, is
\begin{equation}\label{eq:L2-gradient-flow-intro}
    u_t = \varepsilon \Delta u + \frac{1}{\varepsilon} f(u),\quad x\in\Omega\subset\mathbb{R}^d,\; t>0,
\end{equation}
where $0<\varepsilon\ll 1$ is the interfacial width parameter and $f$ is a double-well nonlinearity. This equation is associated with the Ginzburg--Landau energy functional \cite{ginzburg2009theory}:
\begin{equation*}
    E(u)=\int_\Omega \left(\frac{\varepsilon}{2}|\nabla u|^2+\frac{1}{\varepsilon}F(u)\right)  \,\dd x,\quad F'(u)=-f(u).
\end{equation*}
Under periodic boundary conditions or homogeneous Neumann boundary conditions, it satisfies the energy dissipation law
\begin{equation*}
    \frac{d E(u)}{dt} = \left(\frac{\delta E}{\delta u}, u_t \right) = - \norm{u_t}^2  \le 0,\quad t>0,
\end{equation*}
where $(\cdot, \cdot)$ and $\|\cdot\|$ represent the standard $L^2$ inner product and norm.
From a numerical perspective, preserving this intrinsic dissipation structure at the discrete level is highly desirable and beneficial for long-time simulations.

Implicit Runge--Kutta (RK) methods have long been favored for structure preservation. For example, Hairer and Lubich gave a systematic analysis of energy dissipation for implicit methods applied to general gradient flows in \cite{HairerLubich2014IMA_iRK}. It is proved that implicit RK methods of arbitrary order preserve energy dissipation for \eqref{eq:L2-gradient-flow-intro} when the time step $\tau$ is sufficiently small satisfying
\[
    \tau \le C_{\rm ImRK}\varepsilon,
\]
where $C_{\rm ImRK}$ is a constant independent of $\varepsilon$, depending only on the RK coefficients and the upper bound of $f^{\prime}$; see \cite[Theorem 2.1]{HairerLubich2014IMA_iRK} for details.
On the other hand, explicit RK methods are more convenient to implement and computationally more efficient. For example, Fu, Tang, and Yang established in \cite{FuTangYang2022JSC_eRK} both the maximum bound principle and the energy dissipation of explicit RK methods for \eqref{eq:L2-gradient-flow-intro}, provided some positive-definiteness constraints on the RK coefficients and the following CFL-type condition:
\begin{equation*}
    \tau \le C_{\rm ExRK}  \min\!\left\{\varepsilon,\frac{h^2}{\varepsilon}\right\},
\end{equation*}
where $h$ denotes the spatial mesh size, and $C_{\rm ExRK}$ is a constant independent of $\varepsilon$, depending only on the RK coefficients.

Exponential time-differencing Runge–Kutta (ETDRK) methods \cite{HochbruckEtAl1998EI, Cox2002JCP_ETD,KassamTrefethen2005SISC_ETDRK4,HochbruckOstermann2010ActaNum_EI,Du2021SIAMRev_MP4ETD} integrate the linear stiff component exactly via exponential operators and approximate the nonlinear evolution using Runge–Kutta formulas, providing high-order accuracy and improved stability for stiff evolution equations. The construction of high-order ETDRK methods has progressed rapidly, from the classical stiff fourth-order work \cite{Hochbruck2005SINUM_OrderCondition} to the stiff order theory based on exponential B-series \cite{Luan2013ExpBseries}, fifth-order \cite{Luan2014HighOrderExpRK} and sixth-order \cite{LuanAlhsmy2024ExpRK6} methods. 
When applied to gradient flows, the first-order ETDRK1 method has been shown to preserve energy dissipation unconditionally in \cite{Ju2018MathComp_ETD_EnergyStability}.
Later, in \cite{Fu2022JCP_ED4ETD2}, Fu and Yang rewrote the energy difference in quadratic form and established unconditional energy dissipation for the ETDRK2 scheme under some positive-definiteness conditions.
This analysis has been extended to higher-order ETDRK schemes in \cite{Fu2025ChinaMath_ED4HighOrderETD} and to matrix-valued gradient flows in \cite{Quan2024IMA_matrixAC}.
Similar ideas have also been used to establish the energy dissipation of implicit-explicit RK (IMEX-RK) methods \cite{Fu2024MathComp_ED4IMRK} and corrected integrating factor RK (IFRK) methods \cite{Liao2024JCP_correctedIMRK}.
However, we mention that this positive-definiteness framework is somehow restrictive and cannot be applied to all ETDRK methods, and a fourth-order energy-dissipative ETDRK method is still challenging to discover.
Recently, a special class of interpolation-based high-order ETDRK methods has been proved to preserve energy dissipation under some time step-size restrictions in \cite{Quan_25_IMA_ETD_MP}.

In addition, some other energy stable high-order methods for gradient flows have been proposed and extensively studied, such as scalar auxiliary variable (SAV) methods \cite{Shen2019SIAMRev_SAV}, invariant energy quadratization (IEQ) methods \cite{Yang2016JCP}, operator splitting methods \cite{LiDQ2022JCP_SplittingAC_1,LiDQ2022JCP_SplittingAC_2}, and Lagrange multipliers \cite{Cheng2025arXiv_lagrangemultiplier_MPandED}. These approaches can preserve only modified (rather than original) energy dissipation, or involve solving a nonlinear equation at each time step.

Despite the aforementioned advances, existing energy-dissipation analyses for ETDRK schemes either rely on positive-definiteness conditions on the coefficients \cite{Fu2022JCP_ED4ETD2,Fu2025ChinaMath_ED4HighOrderETD} or are limited to special interpolation-based constructions \cite{Quan_25_IMA_ETD_MP}. In this work, we develop a unified framework for a broad class of equilibrium-preserving ETDRK methods. The framework covers standard classical- and stiff-order ETDRK schemes and yields explicitly computable step-size restrictions for original energy dissipation:
$$
    \tau \le C_{\rm ETDRK}\varepsilon,
$$
where the constant $C_{\rm ETDRK}$ is explicitly computable, independent of $\varepsilon$, and depends only on the ETDRK coefficients and the Lipschitz constant of $f$. 
This framework has been applied to some classical ETDRK schemes up to the fifth order, to obtain explicit step-size restrictions for energy dissipation.
Furthermore, for arbitrarily high-order interpolation-based ETDRK methods studied in \cite{Quan_25_IMA_ETD_MP}, the step-size restrictions derived from our new framework are much weaker than the previous results; see comparisons in Table~\ref{tab:tau_max_comparison_Quan}.

The paper is organized as follows.
Section~\ref{sec:preliminary} introduces the ETDRK schemes and an auxiliary partial differential equation (PDE) form.
Section~\ref{sec:ED_ETDRK} establishes the original energy dissipation for general ETDRK methods under mild time-step restrictions.
Section~\ref{sec:explicit_time_step_restriction} computes explicit step-size restrictions for several well-known ETDRK methods of orders from two to five, and for a class of arbitrarily high-order interpolation-based ETDRK schemes.
Section~\ref{sec:numerical_experiments} performs numerical experiments to verify the theoretical results.

\section{Preliminary}\label{sec:preliminary}

In this work, we consider the following $L^2$ gradient flow:
\begin{equation}\label{eq:L2-gradient-flow}
    \left\{
    \begin{aligned}
         & u_{t}      = \varepsilon \hL u + \frac{1}{\varepsilon} f(u), &  & \bm{x} \in \Omega, ~  t \in(0,T], \\
         & u(\bm{x},0)  =u^0({\bm{x}}),                                 &  & \bm{x} \in \bar{\Omega},
    \end{aligned}
    \right.
\end{equation}
on a bounded domain $\Omega\subset\mathbb{R}^d$ ($d=1,2,3$), equipped with homogeneous Neumann or periodic boundary conditions.
Here the linear operator $\hL$ is self-adjoint, negative semidefinite, and has compact resolvent; a typical example is the Laplacian $\Delta$.
In this work, we assume that $f$ is globally Lipschitz continuous with constant $C_L$.
This equation \eqref{eq:L2-gradient-flow} can be viewed as an $L^2$ gradient flow with respect to the energy functional
\begin{equation}\label{eq:E}
    {E(u)=-\frac{\varepsilon}{2}\left(\hL u, u\right) + \frac{1}{\varepsilon}\left(F(u), 1\right) ,}
\end{equation}
where $(\cdot, \cdot)$ represents the standard $L^2$ inner product and $F'(u)=-f(u)$.
The energy dissipation law is formulated as
\begin{equation*}
    \frac{\dd}{\dd t} E(u)=\left(\frac{\delta E(u)}{\delta u}, \frac{\partial u}{\partial t}\right)=-\left\|{\partial_{t} u}\right\|^2 \leq 0, \quad \forall t>0,
\end{equation*}
where $\|\cdot\|$ is the standard $L^2$ norm.

\begin{remark}\label{rem:lipschitz}
This global Lipschitz assumption will be used in the energy-dissipation analysis of ETDRK methods below.
For schemes with a discrete maximum bound principle, the numerical solution remains in an interval on which $f$ is Lipschitz.
Such principles have been established for several ETDRK schemes, see e.g.\ \cite{Du2019SINUM_MP4ETD,Du2021SIAMRev_MP4ETD}.
On the other hand, for some high-order ETDRK schemes without a maximum bound principle, the energy-dissipation analysis still applies whenever the numerical solution remains bounded, as in the experiments of Section~\ref{sec:numerical_experiments}.
\end{remark}

We first state a useful inequality in the following lemma.

\begin{lemma}\label{lem:scalar_lipschitz}
For all \(u,v\in\mathbb R\), one has
\begin{equation}\label{Lipschitz_cond}
  F(u)-F(v)-f(u)(v-u)\ge -\frac{C_L}{2}(v-u)^2.
\end{equation}
\end{lemma}
\begin{proof}
If \(u=v\), then \eqref{Lipschitz_cond} is an equality. If \(u<v\), then for any \(w\in[u,v]\),
\[
|f(w)-f(u)|\le C_L(w-u),
\]
so \(f(w)\ge f(u)-C_L(w-u)\). Since \(f=-F'\), we have
\[
F(u)-F(v)=\int_u^v f(w)\,\dd w
\ge \int_u^v (f(u)-C_L(w-u))\,\dd w
=f(u)(v-u)-\frac{C_L}{2}(v-u)^2.
\]
If \(u>v\), then for any \(w\in[v,u]\),
\[
|f(w)-f(u)|\le C_L(u-w),
\]
so \(f(w)\le f(u)+C_L(u-w)\), and thus
\[
F(v)-F(u)=\int_v^u f(w)\,\dd w
\le \int_v^u (f(u)+C_L(u-w))\,\dd w
=-f(u)(v-u)+\frac{C_L}{2}(v-u)^2.
\]
Hence \eqref{Lipschitz_cond} holds for all \(u,v\in\mathbb R\).
\end{proof}

\subsection{General ETDRK schemes}

ETDRK methods are constructed by applying the variation-of-constants formula to the system \eqref{eq:L2-gradient-flow}, yielding
\begin{equation}\label{eq:L2-gradient-flow-mild}
    u(t)= \e^{\varepsilon\hL(t-t_0)} u(t_0)+ \int_{0}^{t-t_0} \e^{\varepsilon\hL(t-t_0-s)} g(u(t_0+s)) \,\dd s,
\end{equation}
where \[
g \coloneqq \frac{1}{\varepsilon}f.
\]
Let $u^{n}$ denote the numerical solution at the $n$-th time step $t = t_n$. To advance the solution to the next time step $t = t_{n+1}$, ETDRK methods approximate the integral in \eqref{eq:L2-gradient-flow-mild} over the interval $[t_n, t_{n+1}]$. For instance, by replacing $g(u(s))$ with $g(u^n)$ in \eqref{eq:L2-gradient-flow-mild}, one obtains the first-order ETDRK (ETDRK1) scheme:
\begin{equation*}
    u^{n+1}=\e^{\tau\varepsilon\hL} u^{n}+ \tau \varphi_1(\tau\varepsilon\hL) g(u^n),
\end{equation*}
where $\tau$ is the time step size, and the function $\varphi_1$ is defined by \[
    \varphi_1(z) \coloneqq \int_{0}^{1} \e^{(1-s)z} \,\dd s = \left( \e^z-1 \right)z^{-1}.
\]

In general, an $r$-stage ETDRK scheme can be written as follows \cite{Hochbruck2005SINUM_OrderCondition}:
\begin{subequations}\label{eq:ETDRK}
\begin{align}
    v_1     & =u^n, \label{eq:ETDRK_v1}\\
    v_i     & = \e^{c_i \tau \varepsilon \hL} u^n+\tau \sum_{j=1}^{i-1} a_{i j}(\tau \varepsilon \hL) g(v_j), \quad i=2,3, \ldots, r, \label{eq:ETDRK_stage}\\
    u^{n+1} & = \e^{\tau \varepsilon \hL} u^n+\tau\sum_{j=1}^r b_j(\tau \varepsilon \hL) g(v_j), \label{eq:ETDRK_update}
\end{align}
\end{subequations}
where the coefficients \(a_{ij}\) and \(b_j\) are combinations of the \(\varphi\)-functions:
\begin{equation}\label{def:phi_k}
    \begin{aligned}
         & \varphi_0(z) \coloneqq \e^z, \quad \varphi_k(z) \coloneqq \int_0^1 \e^{(1-\theta)z} \frac{\theta^{k-1}}{(k-1)!} \,\dd \theta, \quad k \geq 1, \\
         & \varphi_{k,i}(z) \coloneqq \varphi_k(c_i z), \quad  i=2,3,\ldots,r, \quad k \geq 0.
    \end{aligned}
\end{equation}
The same scheme can be represented by the Butcher tableau \eqref{butcher_tableau}. Here the tableau entries are generally functions; when \(\hL=0\), they reduce to constants and the method recovers a classical explicit RK scheme.
\begin{equation}\label{butcher_tableau}
    \renewcommand{\arraystretch}{1.2}
    \begin{array}{c|cccc}
        0      & ~      & ~      & ~          & ~   \\
        c_2    & a_{21} & ~      & ~          & ~   \\
        \vdots & \vdots & \ddots & ~          & ~   \\
        c_r    & a_{r1} & \cdots & a_{r, r-1} & ~   \\
        \hline
               & b_1    & \cdots & b_{r-1}    & b_r \\
    \end{array}
\end{equation}

For later energy estimates, we impose the following assumptions on the ETDRK scheme:
\begin{equation}\label{assump:ci}
c_i\in[0,1],\qquad 1\le i\le r,
\end{equation}
and
\begin{equation}\label{assump:ab-finite}
\sup_{\tau\ge 0,\varepsilon\ge 0}\vertiii{a_{ij}(\tau \varepsilon \hL)}<\infty,\qquad
\sup_{\tau\ge 0,\varepsilon\ge 0}\vertiii{b_i(\tau \varepsilon \hL)}<\infty,\qquad
1\le i\le r,\ 1\le j<i,
\end{equation}
where the operator norm is defined by
\[
    \vertiii{\mathcal{T}} \coloneqq \sup_{v\in L^2(\Omega)\setminus\{0\}} \frac{\norm{\mathcal{T}v}}{\norm{v}}.
\]
For convenience, we denote the finite suprema as
\begin{equation}\label{def_alpha_beta}
    \alpha_{ij} \coloneqq \sup_{\tau\ge 0,\varepsilon\ge 0}\vertiii{a_{ij}(\tau \varepsilon \hL)},
    \quad \beta_i \coloneqq \sup_{\tau\ge 0,\varepsilon\ge 0}\vertiii{b_i(\tau \varepsilon \hL)},
    \quad 1\le i\le r,\ 1\le j<i.
\end{equation}

An essential property of ETDRK schemes is equilibrium preservation \cite{Hochbruck2005SINUM_OrderCondition}. More precisely, let \(u^*\) be an equilibrium of \eqref{eq:L2-gradient-flow}, i.e., \(\varepsilon\hL u^*+g(u^*)=0\). 
We call the scheme \eqref{eq:ETDRK} equilibrium preserving if for $u^n=u^*$, it holds true that \(v_i=u^*\) for all \(1\le i\le r\) and \(u^{n+1}=u^*\).
A sufficient condition for equilibrium preservation of \eqref{eq:ETDRK}  is 
\begin{align}
     & \sum_{j=1}^r b_j(z) = \frac{\e^z-1}{z} = \varphi_1(z), \label{cond_equilibria_bj}\\
     & \sum_{j=1}^{i-1} a_{i j}(z) = \frac{\e^{c_i z}-1}{z} = c_i\varphi_1(c_i z), \quad i=2,3,\ldots,r. \label{cond_equilibria_aij}
\end{align}
In the stiff-order theory of ETDRK methods \cite[Table 2]{Hochbruck2005SINUM_OrderCondition}, \eqref{cond_equilibria_bj} is the first-order condition, while \eqref{cond_equilibria_aij} is part of the second-order conditions. This means that any stiff-order ETDRK method of order at least two necessarily preserves equilibria.

\subsection{Auxiliary PDE representation}
Since the coefficients $b_j(z)$ are linear combinations of the $\varphi_{k}(z)$, by the definition \eqref{def:phi_k}, we can write $b_j(z)$ as
\begin{equation}\label{eq:b_j_z}
    b_j(z) = \int_0^1 \e^{(1-\theta)z}p_{j}(\theta)  \,\dd \theta, \quad 1 \le j \le r,
\end{equation}
where $p_{j}(\theta)$ are polynomials.

The condition \eqref{cond_equilibria_bj} implies
\begin{equation*}
    \int_0^1 \e^{(1-\theta)z} \sum_{j=1}^r p_{j}(\theta)  \,\dd \theta = \sum_{j=1}^r b_j(z) = \varphi_1(z) = \int_0^1 \e^{(1-\theta)z}   \,\dd \theta,
\end{equation*}
for all \(z\in\mathbb C\). Let
\[
q(\theta)\coloneqq \sum_{j=1}^r p_j(\theta)-1.
\]
Then \(q\) is a polynomial and
\[
\int_0^1 \e^{(1-\theta)z} q(\theta)\,\dd\theta=0,\qquad \forall z\in\mathbb C.
\]
Differentiating at \(z=0\) yields, for every \(m\ge0\),
\[
\int_0^1 (1-\theta)^m q(\theta)\,\dd\theta=0.
\]
Hence \(\int_0^1 r(\theta)q(\theta)\,\dd\theta=0\) for every polynomial \(r\) (since \(\{(1-\theta)^m\}_{m\ge0}\) spans the polynomial space). Taking \(r=q\), we get
\[
\int_0^1 q(\theta)^2\,\dd\theta=0,
\]
so \(q\equiv0\) on \([0,1]\). Therefore,
\begin{equation}\label{sum_pj}
    \sum_{j=1}^r p_{j}(\theta)  = 1.
\end{equation}
Note that
\begin{equation}\label{eq:sum_b_p}
    {\tau \sum_{j=1}^r b_j(\varepsilon\tau \hL) g(v_j)}
    = \tau \int_0^1  \e^{(1-\theta)\varepsilon\tau \hL} \sum_{j=1}^r p_{j}(\theta)\, g(v_j)  \,\dd \theta
    = \int_{0}^{\tau} \e^{(\tau-s)\varepsilon\hL} P(s)  \,\dd s.
\end{equation}
where 
\begin{equation}\label{def_P_s}
    P(s) \coloneqq \sum_{j=1}^{r} p_{j}\left(\frac{s}{\tau}\right)  g(v_j), \quad s \in [0,\tau].
\end{equation}

This identity suggests a \emph{PDE representation} of the ETDRK scheme. Consider the auxiliary PDE:
\begin{equation} \label{eq:w}
    \left\{
    \begin{aligned}
         & \partial_{s} w^{n}  =\varepsilon\hL w^{n}+ P(s), &  & \bm{x} \in \Omega ,~ s \in(0, \tau], \\
         & w^{n}(\bm{x},0) =u^n,                            &  & \bm{x} \in \bar\Omega ,
    \end{aligned}
    \right.
\end{equation}
subject to appropriate boundary conditions.
According to \eqref{eq:sum_b_p}, the numerical solution of the ETDRK scheme satisfies
\begin{equation}\label{eq:ETDRK_solution}
    {u^{n+1} = \e^{\varepsilon\tau\hL} u^n+\tau\sum_{j=1}^r b_j(\varepsilon\tau \hL) g(v_j) = \e^{\varepsilon\tau\hL} u^n+ \int_{0}^{\tau} \e^{(\tau-s)\varepsilon\hL} P(s)  \,\dd s = w^{n}(\tau),}
\end{equation}
which is exactly the solution of \eqref{eq:w} at time $s = \tau$. Thus, the ETDRK scheme can be interpreted as solving the auxiliary problem \eqref{eq:w} over the interval $[0,\tau]$, which will be helpful for our subsequent energy analysis.

\section{Energy dissipation for general ETDRK schemes}\label{sec:ED_ETDRK}

In this section, we present a unified framework to prove that a broad class of equilibrium-preserving ETDRK methods preserve the original energy dissipation under mild time step-size restrictions.
The time derivative of the energy functional along the solution $w^{n}(s)$ of the auxiliary problem \eqref{eq:w} is given by
\begin{equation}\label{eq:dEwdt}
    \begin{aligned}
        \frac{\mathrm{d}}{\mathrm{d} s} E(w^{n}(s))
         & = \left(\frac{\delta E(w^{n})}{\delta w^{n}} ,\partial_s w^{n}\right)     \\
         & = -\left(\varepsilon\hL w^{n} + g(w^{n}),\partial_s w^{n}\right)          \\
         & = -\norm{\partial_s w^{n}}^2+\left(P(s)-g(w^{n}),\partial_s w^{n}\right),
    \end{aligned}
\end{equation}
leading to the following lemma.
\begin{lemma}\label{lemma:energy_decay}
  Assume \(f\) is globally Lipschitz with constant \(C_L\).
  Let \(\{u^{n}\}_{n\ge 0}\) be the numerical solution to \eqref{eq:L2-gradient-flow} obtained from the ETDRK scheme \eqref{eq:ETDRK}.
    Then
    \begin{equation}\label{ieq:aim}
        E(u^{n+1}) -E(u^{n})  \le \frac{1}{2}\left(-\frac{1}{\tau} + \frac{C_L}{\varepsilon}\right) \norm{u^{n+1}-u^{n}}^2  + \frac{1}{2}\int_{0}^{\tau} \norm{P(s)-g(u^n)}^2 \,\dd s,
    \end{equation}
    where $P(s)$ is defined in \eqref{def_P_s}.
\end{lemma}

\begin{proof}
    By \eqref{eq:dEwdt}, the energy difference between $u^n$ and $u^{n+1}$ is given by
    \begin{equation}\label{eq:E_1-E_0}
        \begin{aligned}
            E(u^{n+1}) - E(u^{n})
             & = -\int_{0}^{\tau} \norm{\partial_s w^{n}}^2  \,\dd s + \int_{0}^{\tau}  \left( P(s) - g(w^{n}), \partial_s w^{n}\right) \,\dd s \\
             &= -\int_{0}^{\tau} \norm{\partial_s w^{n}}^2  \,\dd s + \int_{0}^{\tau}  \left( P(s) - g(u^{n}), \partial_s w^{n}\right) \,\dd s \\
             & \quad - \int_{0}^{\tau}  \left( g(w^{n}) - g(u^{n}), \partial_s w^{n}\right) \,\dd s. 
        \end{aligned}
    \end{equation}

    For any fixed $\bm{x} \in \Omega$, using \eqref{Lipschitz_cond} in Lemma~\ref{lem:scalar_lipschitz} with $u = u^n$ and $v = u^{n+1}$, we have
    \begin{equation}\label{ieq:E_1-E_0_third_1}
        \begin{aligned}
            \int_{0}^{\tau} \left[ f(w^{n}(s)) - f(u^{n}) \right]   \partial_s w^{n} \,\dd s
             & = F(u^{n}) - F(u^{n+1}) - f(u^{n}) (u^{n+1}-u^{n}) \\
             & \ge  -\frac{C_L}{2}\left(u^{n+1}-u^{n}\right)^{2},
        \end{aligned}
    \end{equation}
    where we have used the fact that $w^{n}(\tau)=u^{n+1}$ by \eqref{eq:ETDRK_solution} and $w^{n}(0)=u^{n}$. 
    The inequality \eqref{ieq:E_1-E_0_third_1} yields the bound for the third term in \eqref{eq:E_1-E_0}:
    \begin{equation}\label{ieq:E_1-E_0_third}
        \begin{aligned}
             -\int_{0}^{\tau}  \left( g(w^{n}) - g(u^{n}), \partial_s w^{n}\right) \,\dd s                     
             \le  \frac{C_L}{2\varepsilon} \norm{u^{n+1}-u^{n}}^2 .
        \end{aligned}
    \end{equation}
For the second term in \eqref{eq:E_1-E_0}, we have
    \begin{equation}\label{ieq:E_1-E_0_second}
      \begin{aligned}
           \int_{0}^{\tau}  \left( P(s) - g(u^{n}), \partial_s w^{n}\right) \,\dd s          
           \le \frac{1}{2} \left(  \int_{0}^{\tau} \norm{P(s)-g(u^n)}^2 \,\dd s + \int_{0}^{\tau} \norm{\partial_s w^{n}}^2 \,\dd s \right).
      \end{aligned}
  \end{equation}

    Next, we estimate $ \int_{0}^{\tau} \norm{\partial_s w^{n}}^2 \,\dd s$ using the Cauchy--Schwarz inequality:
    \[
        \tau^{\frac{1}{2}}\left(\int_{0}^{\tau} \norm{\partial_s w^{n}}^2 \,\dd s\right)^{\frac{1}{2}} \geq \int_{0}^{\tau} \norm{\partial_s w^{n}} \,\dd s \geq \norm{\int_{0}^{\tau} {\partial_s w^{n}} \,\dd s} = \norm{u^{n+1} - u^{n}},
    \]
    which gives
    \begin{equation}\label{ieq:E_1-E_0_first}
        \int_{0}^{\tau} \norm{\partial_s w^{n}}^2  \,\dd s \geq \frac{1}{\tau} \norm{u^{n+1} - u^{n}}^2.
    \end{equation} 

    Finally, substituting \eqref{ieq:E_1-E_0_third}, \eqref{ieq:E_1-E_0_second} and \eqref{ieq:E_1-E_0_first} back into \eqref{eq:E_1-E_0} yields
    \begin{equation*}
        E(u^{n+1}) - E(u^{n}) \le \left(-\frac{1}{2\tau} + \frac{C_L}{2\varepsilon}\right)  \norm{u^{n+1}-u^{n}}^2 + \frac{1}{2} \int_{0}^{\tau} \norm{P(s)-g(u^n)}^2 \,\dd s.
    \end{equation*}
    This completes the proof.
\end{proof}

From \eqref{ieq:aim}, preserving energy dissipation reduces to proving the following inequality:
\begin{equation*}
    \int_{0}^{\tau} \norm{P(s)-g(u^n)}^2 \,\dd s \le \left(\frac{1}{\tau}-\frac{C_L}{\varepsilon}\right) \norm{u^{n+1}-u^{n}}^{2},
\end{equation*}
under some mild time step-size restrictions.
Since all $p_j$ sum up to one (by condition \eqref{sum_pj}), we have the decomposition
\begin{equation*}
    P(s) - g(u^n) = \sum_{j=1}^{r} p_{j}\left(\frac{s}{\tau}\right) \left[g(v_j)-g(u^n)\right].
\end{equation*}
Applying the Lipschitz condition \eqref{Lipschitz_cond} for $g$ yields
\begin{equation}\label{ieq:norm_Pi_m_g}
    \begin{aligned}
        \left\| P(s) -g(u^n) \right\|
         & \le \sum_{j=1}^{r}\left\vert p_{j}\left(\frac{s}{\tau}\right)  \right\vert  \left\| g(v_j)-g(u^n) \right\|                    \\
         & \le \frac{C_L}{\varepsilon} \sum_{j=1}^{r}\left\vert p_{j}\left(\frac{s}{\tau}\right)  \right\vert  \left\| v_j-u^n \right\|.
    \end{aligned}
\end{equation}
Thus, to control $\norm{P(s)-g(u^n)}$ by $\norm{u^{n+1}-u^{n}}$, it suffices to control each $\norm{v_j-u^n}$ by $\norm{u^{n+1}-u^n}$ for $1 \le j \le r$. 

\begin{lemma}\label{lem:norm_Phi_c}
  For any $c \in [0,1]$, define
  \begin{equation}\label{def_Phi_c}
    \Phi_c \coloneqq c\tau\,\varphi_1(c\tau \varepsilon \hL)\, ( \varepsilon \hL u^n + g(u^n) ) = (e^{c\tau \varepsilon \hL} -1)\, u^n + c\tau\,\varphi_1(c\tau \varepsilon \hL)\,g(u^n) ,
  \end{equation}
  with $\varphi_1$ as in \eqref{def:phi_k}.
  Then
  \begin{equation*}
      \norm{\Phi_c} \le \norm{\Phi_1}.
  \end{equation*}
\end{lemma}
\begin{proof} 
  We show that $\norm{\Phi_c}$ is monotonically increasing in $c$. If $c=0$, then $\Phi_0=0$, so $\norm{\Phi_0}=0\le\norm{\Phi_1}$. Hence we assume \(c\in(0,1]\). Define
  \[
      h(c,\mu)\coloneqq \bigl(c\tau\,\varphi_1(c\tau\mu)\bigr)^2,
  \]
  which is well-defined at \(\mu=0\) since \(\varphi_1(0)=1\).
  For fixed \(\mu\in\RR\), \(h(c,\mu)\) is increasing in \(c\). To see this, if \(\mu\neq0\), then \(h(c,\mu)=(\e^{c\tau\mu}-1)^2/\mu^2\) and
  \[
      \partial_c h(c,\mu)=\frac{2\tau(\e^{c\tau\mu}-1)\e^{c\tau\mu}}{\mu}\ge0,
  \]
  because \((\e^{c\tau\mu}-1)/\mu\ge0\). If \(\mu=0\), \(h(c,0)=c^2\tau^2\) is clearly increasing.

  Let \((\mu_j,\phi_j)_{j=1}^\infty\) denote the eigenpairs of \(\varepsilon\hL\). By the compact-resolvent assumption, \(\varepsilon\hL\) has a discrete spectrum \(\{\mu_j\}\subset(-\infty,0]\) and \(\{\phi_j\}\) forms a complete orthonormal basis of \(L^2(\Omega)\). Then, for any \(v\in L^2(\Omega)\),
  \begin{equation*}
      \begin{aligned}
          \left\|c\tau\,\varphi_1(c\tau\varepsilon \hL)\, v\right\|^2
           & = \sum_{j=1}^\infty h(c,\mu_j)\,|(v,\phi_j)|^2       \\
           & \le \sum_{j=1}^\infty h(1,\mu_j)\,|(v,\phi_j)|^2      \\
           & = \norm{\tau\,\varphi_1(\tau\varepsilon \hL)\, v}^2.
      \end{aligned}
  \end{equation*}
  Taking \(v=\varepsilon \hL u^n+g(u^n)\) yields the desired inequality.
\end{proof}

\begin{lemma}\label{lem:Ci} 
  Assume \(f\) is globally Lipschitz with constant \(C_L\).
  Consider the ETDRK scheme \eqref{eq:ETDRK} under assumptions 
  \eqref{assump:ci}--\eqref{assump:ab-finite}. 
  Define \(\{C_i(x)\}_{i=1}^r\) recursively by
  \begin{equation}\label{def_Ci}
      C_1(x) \coloneqq 0, \qquad  
      C_i(x) := 1 + x
      \sum_{j=1}^{i-1} \alpha_{ij} C_j(x), 
      \quad 2 \le i \le r,
      \quad x \ge 0.
  \end{equation}
  where \(\alpha_{ij}\) is defined in \eqref{def_alpha_beta}.
  Then, 
  \begin{equation}\label{ieq:Ci}
      \norm{v_i - u^{n}} 
      \le C_i\left(\frac{C_L\tau}{\varepsilon}\right)\norm{\Phi_1}, \quad 1 \le i \le r,
  \end{equation}
 and
  \begin{equation}\label{ieq:D_r1}
      \norm{u^{n+1} - u^{n}}
      \ge \left[ 1 - \frac{C_L\tau}{\varepsilon}
      \sum_{i=1}^{r} \beta_i C_i \left(\frac{C_L\tau }{\varepsilon}\right) \right] \norm{\Phi_1},
  \end{equation}
  where $v_i$ is the $i$-th stage value of the ETDRK scheme \eqref{eq:ETDRK}, \(\beta_i\) is defined in \eqref{def_alpha_beta}, and $\Phi_1$ is defined in \eqref{def_Phi_c} with $c=1$.
  \end{lemma}

\begin{proof}
    For \(i=2,\dots,r\), by \eqref{eq:ETDRK_stage} and \eqref{cond_equilibria_aij}, we have
    \begin{equation}\label{eq:vi_un_rewrite}
      \begin{aligned}
          v_i - u^{n}
          & = (\e^{c_i\tau\varepsilon\hL} - \hI) u^n+\tau \sum_{j=1}^{i-1} a_{i j}(\tau\varepsilon\hL) \left(g(v_j)-g(u^n)\right) + c_i\tau \varphi_1(c_i\tau \varepsilon\hL) g(u^n) \\
          & = \Phi_{c_i} + \tau \sum_{j=1}^{i-1} a_{i j}(\tau\varepsilon\hL) \left(g(v_j)-g(u^n)\right).
      \end{aligned}
    \end{equation}
    Similarly, by \eqref{eq:ETDRK_update} and \eqref{cond_equilibria_bj}, we have
    \begin{equation}\label{eq:unplus1_un_rewrite}
      u^{n+1} - u^{n}
      = \Phi_1 +\tau\sum_{j=1}^r b_j(\tau\varepsilon\hL) \left( g(v_j)-g(u^n) \right).
    \end{equation}

    By \eqref{eq:ETDRK_v1}, $v_1=u^{n}$, so the statement holds trivially for $i=1$. For $2 \le i \le r$, taking the $L^{2}$-norm in \eqref{eq:vi_un_rewrite}, we have
    \begin{align*}
            \norm{v_i - u^{n}}
             & \leq \norm{\Phi_{c_i}} + \tau \sum_{j=1}^{i-1} \vertiii{a_{i j}(\tau\varepsilon\hL)} \norm{g(v_j)-g(u^n)} \\
             & \leq \norm{\Phi_1} + \frac{C_L\tau}{\varepsilon} \sum_{j=1}^{i-1} \alpha_{ij} \norm{v_j - u^{n}},          \\
    \end{align*}
    where we have applied Lemma~\ref{lem:norm_Phi_c} and the assumption $c_i\in[0,1]$ in the last step.
    The inequality \eqref{ieq:Ci} then follows from a straightforward induction.

    Next, we estimate $\norm{u^{n+1}-u^{n}}$. Using \eqref{eq:unplus1_un_rewrite}, we have
    \begin{align*}
            \norm{u^{n+1} - u^{n}}
             & =\left\| \Phi_1 +\tau\sum_{j=1}^r b_j(\tau\varepsilon\hL) ( g(v_j)-g(u^n) ) \right\|  \\
             & \geq \norm{\Phi_1} - \tau \sum_{j=1}^r \vertiii{b_j(\tau\varepsilon\hL)} \norm{g(v_j) - g(u^{n})} \\
             & \geq \norm{\Phi_1} - \tau \frac{C_L}{\varepsilon} \sum_{j=1}^r \beta_j \norm{v_j - u^{n}}   \\
             & \geq \left[1-\frac{C_L\tau}{\varepsilon} \sum_{j=1}^{r} C_j\!\left(\tau \frac{C_L}{\varepsilon}\right) \beta_j \right] \norm{\Phi_1}.
    \end{align*}
    This completes the proof.
\end{proof}

\begin{lemma}\label{lem:explicit_lb}
For \(x\in[0,1]\), define
\begin{equation}\label{def_I}
    I(x) \coloneqq \left[\int_{0}^{1}\left(\sum_{i=1}^{r} C_i(x)\abs{p_i(\theta)}\right)^2\,\dd\theta\right]^{\frac12},
\end{equation} and
\begin{equation*}
    S(x) \coloneqq \sum_{i=1}^{r} C_i(x)\beta_i,
\end{equation*}
where \(C_i(x)\) refers to the functions defined in \eqref{def_Ci}, \(p_i(\theta)\) is defined in \eqref{eq:b_j_z}, and \(\beta_i\) is defined in \eqref{def_alpha_beta}.
For \(x\in[0,1)\) and $r\geq 2$, define
\begin{equation}\label{def_H}
    H(x) \coloneqq \frac{x}{\sqrt{1-x}}I(x)+xS(x)-1.
\end{equation}
The function \(H\) admits at least one positive root. Define
\begin{equation}\label{def_x0}
    x_0 \coloneqq \min\{x>0:\,H(x)=0\}.
\end{equation}
Then $x_0$ depends only on the ETDRK method coefficients, and satisfies
\begin{equation*}
    x_0 \ge \min\!\left\{\frac{1}{3I(1)},\frac34\right\}.
\end{equation*}
Moreover, $H(x)\le0$ for all $x\in[0,x_0]$.
\end{lemma}
\begin{proof}
By \eqref{def_Ci}, for \(i=1\) we have \(C_1(x)=0=C_1(1)\). For \(i\ge2\) and every \(x\in[0,1]\),
\[
    C_i(x)=1+x\sum_{j=1}^{i-1}\alpha_{ij}C_j(x)\le 1+\sum_{j=1}^{i-1}\alpha_{ij}C_j(x),
\]
so an induction on \(i\) gives
\begin{equation*}
    C_i(x)\le C_i(1),\qquad i=1,\dots,r,\ \ x\in[0,1].
\end{equation*}
Hence \(I(x)\le I(1)\) and \(S(x)\le S(1)\), and therefore
\begin{equation}\label{ieq:H_prebound}
    H(x)\le \frac{x}{\sqrt{1-x}}I(1)+xS(1)-1,\qquad x\in[0,1).
\end{equation}

Next, we prove \(S(1)\le I(1)\). By \eqref{def_alpha_beta} and \eqref{ieq:pi_opnorm}, we have
\[
    \beta_i\le \int_0^1 \abs{p_i(\theta)}\,\dd\theta.
\]
Thus
\begin{equation}\label{ieq:S_le_I}
\begin{aligned}
    S(1)
    &= \sum_{i=1}^{r} C_i(1)\beta_i
    \le \sum_{i=1}^{r} C_i(1)\int_0^1\abs{p_i(\theta)}\,\dd\theta \\
    &= \int_0^1\sum_{i=1}^{r} C_i(1)\abs{p_i(\theta)}\,\dd\theta
    \le \left[\int_0^1\left(\sum_{i=1}^{r} C_i(1)\abs{p_i(\theta)}\right)^2\,\dd\theta\right]^{\frac12}
    = I(1).
\end{aligned}
\end{equation}
Combining \eqref{ieq:H_prebound} and \eqref{ieq:S_le_I}, we obtain
\begin{equation*}
    H(x)\le x\left(\frac{1}{\sqrt{1-x}}+1\right)I(1)-1.
\end{equation*}
If \(x\le 3/4\), then \(\frac{1}{\sqrt{1-x}}\le 2\), so
\[
    H(x)\le 3xI(1)-1.
\]
Therefore \(H(x)\le0\) whenever
\[
    0\le x\le \min\!\left\{\frac{1}{3I(1)},\frac34\right\}.
\]
In particular, \(H(0)=-1<0\), while \(\lim_{x\to1^-}H(x)=+\infty\) since $I(1)>0$. By continuity, the set
\[
    \{x>0:\,H(x)=0\}
\]
is nonempty, so \(x_0\) is well-defined. 
Moreover, by \eqref{def_I}--\eqref{def_H} together with \eqref{def_Ci}, \eqref{def_alpha_beta}, and \eqref{eq:b_j_z}, the function \(H\) depends only on the ETDRK coefficients \(\alpha_{ij}\), \(\beta_i\), and \(\{p_i\}\). Thus, \(x_0\) only depends on the ETDRK method coefficients.
\end{proof}
It is well-known that the ETDRK1 method preserves unconditionally the original energy dissipation law (see for example \cite{Fu2022JCP_ED4ETD2}). In the following result, we only consider the case of $r\geq 2$.

\begin{theorem}\label{main_thm}
  Assume \(f\) is globally Lipschitz with constant \(C_L\).
  Let \(\{u^{n}\}_{n\ge 0}\) be the numerical solution to \eqref{eq:L2-gradient-flow} obtained from the $r$-stage ETDRK scheme \eqref{eq:ETDRK} with $r\geq 2$ under assumptions \eqref{assump:ci}--\eqref{assump:ab-finite} and \eqref{cond_equilibria_bj}--\eqref{cond_equilibria_aij}.
  Then
  \[
      E(u^{n+1}) \le E(u^{n}),
  \]
  under the step-size restriction
  \[
      \tau \le \tau_{\max} \coloneqq \frac{\varepsilon}{C_L}x_0,
  \]
  where \(x_0\) is defined in \eqref{def_x0}, and hence depends only on the ETDRK method coefficients and is bounded away from zero.
\end{theorem}

\begin{proof} 
    By Lemma~\ref{lemma:energy_decay}, energy dissipation holds if $\tau<{\varepsilon}/{C_L}$ and
    \begin{equation}\label{aim_ieq}
        \tau^{\frac{1}{2}}\left(\int_{0}^{\tau} \|P(s)-g(u^n)\|^2 \,\dd s\right)^{\frac{1}{2}}  
        \le 
        \left(1-\frac{C_L\tau}{\varepsilon}\right)^{\frac{1}{2}}
        \norm{u^{n+1}-u^n}.
    \end{equation}
    Combining \eqref{ieq:norm_Pi_m_g}, \eqref{ieq:Ci}, and \eqref{ieq:D_r1}, we bound both sides of \eqref{aim_ieq} as follows:
    \begin{equation*}
        \begin{aligned}
            \text{LHS} 
            & \le 
            \frac{C_L\tau}{\varepsilon}  \left[\int_{0}^{1} \left(\sum_{i=1}^{r}C_i\!\left(\frac{C_L\tau}{\varepsilon}\right)| p_{i}(\theta) |   \right)^{2}  \,\dd \theta \right]^{\frac{1}{2}} \norm{\Phi_1} \\
            & = \frac{C_L\tau}{\varepsilon} I \left(\frac{C_L \tau}{\varepsilon}\right) \norm{\Phi_1}, \\
            \text{RHS} 
            & \ge 
            \left(1-\frac{C_L\tau}{\varepsilon}\right)^{\frac{1}{2}} \left( 1 - \frac{C_L\tau}{\varepsilon} \sum_{i=1}^{r} C_i\left(\frac{C_L\tau}{\varepsilon}\right) \beta_i \right) \norm{\Phi_1}\\ 
            &= \left(1-\frac{C_L\tau}{\varepsilon}\right)^{\frac{1}{2}} \left( 1 - \frac{C_L\tau}{\varepsilon} S \left(\frac{C_L \tau}{\varepsilon}\right) \right) \norm{\Phi_1},
        \end{aligned}
    \end{equation*}
    where \(I(x)\) and \(S(x)\) are defined in Lemma~\ref{lem:explicit_lb}.
    Therefore, $E(u^{n+1})\le E(u^n)$ holds provided that \(H\!\left(\frac{C_L\tau}{\varepsilon}\right)\le 0\), where \(H\) is defined in \eqref{def_H}.
By Lemma~\ref{lem:explicit_lb}, \(H(x)\le0\) for all \(x\in[0,x_0]\). Setting \(x=\frac{C_L\tau}{\varepsilon}\), we obtain that
\[
   \tau \le \frac{\varepsilon}{C_L}x_0
\]
implies $E(u^{n+1})\le E(u^n)$.
\end{proof}

\begin{remark}\label{rem:explicit_step_from_x0_lb}
  Lemma~\ref{lem:explicit_lb} ensures that \(x_0\) is bounded away from zero:
  \[
      x_0 \ge 
      \min\left\{
          \frac{1}{3I(1)},
          \frac{3}{4}
      \right\}
      > 0.
  \]
  Substituting this lower bound into Theorem~\ref{main_thm} yields the following explicit sufficient step-size condition for energy dissipation:
  \begin{equation*}
      \tau \le 
      \frac{\varepsilon}{C_L}
      \min\left\{
          \frac{1}{3 I(1)},
          \frac{3}{4}
      \right\},
  \end{equation*}
  where \(I(x)\) is defined in \eqref{def_I}.
  \end{remark}

  \begin{remark}\label{rem:fully_discrete}
    For a fully discrete scheme, $\hL$ is replaced by a spatially discretized operator $\hL_h$.
    If $\hL_h$ is still self-adjoint and negative semidefinite in the discrete $L^2$ space, then $\{e^{t\varepsilon\hL_h}\}$ is a contraction semigroup with $\|e^{t\varepsilon\hL_h}\|\le 1$.
    Consequently, the energy dissipation result stated in Theorem~\ref{main_thm} still holds, with $\hL$ replaced by $\hL_h$ throughout the argument.
    \end{remark}

\section{Explicit step-size restriction for various ETDRK methods}\label{sec:explicit_time_step_restriction}

In this section, we compute explicit step-size restrictions for several representative methods, which are fundamentally classified by the type of order conditions \cite{Hochbruck2005SINUM_OrderCondition}.
Classical-order ETDRK schemes are derived from standard ODE Taylor expansions~\cite{Cox2002JCP_ETD}. However, Hochbruck and Ostermann demonstrated that they may suffer from severe order reduction for stiff PDEs, because classical error bounds implicitly require the time derivatives of the nonlinear source to satisfy spatial boundary conditions. To overcome this, they introduced stiff order conditions \cite[Table~2]{Hochbruck2005SINUM_OrderCondition} to achieve full-order convergence, which impose additional constraints on the method's coefficients.
Our goal here is not to compare their accuracy, but to demonstrate that our framework is suitable for both classical-order and stiff-order methods, providing energy dissipation step-size restrictions regardless of how their coefficients were derived.

The quantities \(\alpha_{ij}\) and \(\beta_i\) in \eqref{def_alpha_beta} are suprema, which may be difficult to evaluate.
The analysis of Theorem~\ref{main_thm} remains valid if they are replaced by any upper bounds \(\halp_{ij}\ge\alpha_{ij}\) and \(\hbet_i\ge\beta_i\).
In the computations below we use the integral bounds
\begin{equation}\label{ieq:pi_opnorm}
  \vertiii{\int_0^1 \e^{(1-\theta)\tau\varepsilon\hL}p_i(\theta)\,\dd\theta}
  \le
  \int_0^1|p_i(\theta)|\,\dd\theta,\qquad 1\le i\le r,
\end{equation}
and the analogous estimate for each internal coefficient $a_{ij}$.
These bounds are in general not optimal: if the polynomial kernel of \(b_i\) or \(a_{ij}\) changes sign, the integral of the absolute value may strictly exceed the corresponding operator-norm supremum.
We adopt them because they are convenient to compute and admit a uniform implementation; the corresponding unified procedure is given in Algorithm~\ref{alg:htaumax}.
The resulting step-size restriction may therefore be more conservative than the exact \(\tau_{\max}\).
We denote by \(\hxzero\) the smallest positive root of \(H\) in \eqref{def_H} evaluated with \(\halp_{ij}\) and \(\hbet_i\), and we define $\htaumax\coloneqq \varepsilon \hxzero/C_L$.
Thus the quantity computed in this section is \(\htaumax\) rather than \(\tau_{\max}\).

\begin{algorithm}[ht]
\caption{Computation of \(\htaumax\) from ETDRK coefficients}
\label{alg:htaumax}
\begin{algorithmic}[1]
\Require Butcher coefficients of an $r$-stage ETDRK scheme \eqref{eq:ETDRK}
\Ensure $\htaumax=\varepsilon\,\hxzero/C_L$
\For{$i \leftarrow 2$ to $r$} \Comment{skip $i=1$ since $C_1\equiv 0$}
    \State expand $b_i(z)=\sum_{k\ge 1}\gamma_{ik}\varphi_k(z)$
    \State $p_i(\theta)\leftarrow\sum_{k\ge 1}\gamma_{ik}\,\theta^{k-1}/(k-1)!$
    \State $\hbet_i\leftarrow\int_0^1 |p_i(\theta)|\,\dd\theta$
\EndFor
\For{$i \leftarrow 3$ to $r$}
    \For{$j \leftarrow 2$ to $i-1$} \Comment{skip column $j=1$ since $C_1\equiv 0$}
        \State expand $a_{ij}(z)=\sum_{c}\sum_{k\ge 1}\gamma_{ij,ck}\,\varphi_k(c z)$
        \State $q_{ij,c}(\theta)\leftarrow\sum_{k\ge 1}\gamma_{ij,ck}\,\theta^{k-1}/(k-1)!$
        \State $\halp_{ij}\leftarrow\sum_{c}\int_0^1 |q_{ij,c}(\theta)|\,\dd\theta$
    \EndFor
\EndFor
\State $C_2(x)\leftarrow 1$
\For{$i \leftarrow 3$ to $r$}
    \State $C_i(x)\leftarrow 1+x\sum_{j=2}^{i-1}\halp_{ij}C_j(x)$
\EndFor
\State $I(x)\leftarrow\bigl[\int_0^1\bigl(\sum_{i=2}^r C_i(x)|p_i(\theta)|\bigr)^2\,\dd\theta\bigr]^{1/2}$
\State $S(x)\leftarrow\sum_{i=2}^r C_i(x)\,\hbet_i$
\State $H(x)\leftarrow \frac{x}{\sqrt{1-x}}I(x)+x S(x)-1$
\State $\hxzero\leftarrow$ smallest root of $H(x)=0$ in $(0,1)$
\State \Return $\varepsilon\,\hxzero/C_L$
\end{algorithmic}
\end{algorithm}

\subsection{Second-order methods} \label{subsec:second_order}

We begin our analysis with second-order methods. According to the stiff order theory \cite{Hochbruck2005SINUM_OrderCondition}, an ETDRK method of order two requires at least two internal stages. Solving the corresponding stiff-order conditions directly yields a one-parameter family of schemes \cite{Hochbruck2005SINUM_OrderCondition}, parameterized by $c_2 \in (0, 1]$:
\begin{equation}\label{tab:HO_ETDRK2}
  \renewcommand{\arraystretch}{1.2}
  \begin{array}{c|cc}
    0      &                        & \\
    c_2    & c_2 \varphi_{1,2}      & \\
    \hline & \varphi_1-\frac{1}{c_2}\varphi_2 & \frac{1}{c_2}\varphi_2
  \end{array}.
\end{equation}
For $c_2\in(0,1]$, Theorem~\ref{main_thm} gives $\htaumax(c_2)=\frac{\varepsilon}{C_L}\hxzero(c_2)$. We maximize $\htaumax$ over $c_2$ to obtain the optimal method.

For \eqref{tab:HO_ETDRK2}, Algorithm~\ref{alg:htaumax} starts from \eqref{eq:b_j_z}.
Since the first stage is skipped ($C_1\equiv0$), only
\[
  b_2=\frac{1}{c_2}\varphi_2
\]
enters the computation, hence
\[
  p_2(\theta)=\frac{\theta}{c_2},\qquad
  \hbet_2=\int_0^1\bigl|p_2(\theta)\bigr|\,\dd\theta=\frac{1}{2c_2}.
\]
The only nonzero internal coefficient is $a_{21}$, which belongs to the first column $j=1$.
This column is skipped in Algorithm~\ref{alg:htaumax}, so no $\halp_{ij}$ arises.
With $C_2(x)=1$,
\[
  I(x)=\Biggl[\int_0^1\bigl(C_2(x)\bigl|p_2(\theta)\bigr|\bigr)^2\,\dd\theta\Biggr]^{\frac12}
  =\frac{1}{\sqrt{3}\,c_2},\qquad
  S(x)=C_2(x)\,\hbet_2=\frac{1}{2c_2},
\]
and \eqref{def_H} becomes
\begin{equation*}
  H(x)=\frac{x}{c_2}\left(\frac{1}{\sqrt{3}\sqrt{1-x}}+\frac{1}{2}\right)-1,
  \qquad x\in[0,1).
\end{equation*}
The equation $H(x)=0$ is equivalent to
\begin{equation}\label{eq:HO_ETDRK2_x0}
  \frac{x}{\sqrt{3}\sqrt{1-x}}+\frac{x}{2}=c_2.
\end{equation}
Since the left-hand side of \eqref{eq:HO_ETDRK2_x0} is strictly increasing in $x$ on $(0,1)$, the root $\hxzero(c_2)$ in \eqref{def_x0} is uniquely determined and strictly increasing in $c_2$.
The maximum over $c_2\in(0,1]$ is attained at $c_2=1$, where $\hxzero(1)=\frac{2}{3}$ and $\htaumax(1)=\frac{2}{3}\,\frac{\varepsilon}{C_L}$.

The choice $c_2=1$ corresponds to the classical ETDRK2 method. This method was proved in \cite{Fu2022JCP_ED4ETD2} to preserve original energy dissipation unconditionally when combined with a stabilization technique. 
In addition, this stabilized ETDRK2 method was proved to preserve the maximum bound principle for Allen--Cahn-type equations~\cite{Du2021SIAMRev_MP4ETD,Du2019SINUM_MP4ETD}.

\subsection{Third-order methods} \label{subsec:third_order}

A third-order ETDRK method requires at least three stages \cite{Hochbruck2005SINUM_OrderCondition}.
In this subsection we evaluate the step-size restriction $\htaumax$ from Theorem~\ref{main_thm} for three representative three-stage schemes: the stiff-order Hochbruck--Ostermann ETDRK3 method \cite{Hochbruck2005SINUM_OrderCondition}, and the classical-order Cox--Matthews \cite{Cox2002JCP_ETD} and Celledoni--Marthinsen--Owren \cite{CelledoniMarthinsenOwren2003} ETDRK3 methods.
We record their Butcher tableaux below.

The Hochbruck--Ostermann ETDRK3 family depends on a single parameter $c_2\in(0,1]$:
\begin{equation}\label{tab:HO_ETDRK3_1}
  \renewcommand{\arraystretch}{1.2}
  \begin{array}{c|ccc}
    0      &                                         &                                    & \\
    c_2    & c_2\varphi_{1,2}                         &                                    & \\
    \frac{2}{3} & \frac{2}{3}\varphi_{1,3}-\frac{4}{9c_2}\varphi_{2,3} & \frac{4}{9c_2}\varphi_{2,3} & \\
    \hline
           & \varphi_1-\frac{3}{2}\varphi_2            & 0                                  & \frac{3}{2}\varphi_2
  \end{array}
\end{equation}
The Cox--Matthews ETDRK3 scheme is
\begin{equation}\label{tab:CM_ETDRK3}
    \renewcommand{\arraystretch}{1.2}
    \begin{array}{c|ccc}
        0           &                                      &                         &                          \\
        \frac{1}{2} & \frac{1}{2} \varphi_{1,2}            &                         &                          \\
        1           & -\varphi_{1,3}                       & 2 \varphi_{1,3}         &                          \\
        \hline      & \varphi_1 - 3\varphi_2 + 4 \varphi_3 & 4\varphi_2 -8 \varphi_3 & -\varphi_2 + 4 \varphi_3
    \end{array}
\end{equation}
Finally, Celledoni--Marthinsen--Owren ETDRK3 has tableau
\begin{equation}\label{tab:CMO_ETDRK3}
    \renewcommand{\arraystretch}{1.2}
    \begin{array}{c|ccc}
      0 & & \\
      \frac{1}{3} & \frac{1}{3}\varphi_{1,2} & \\
      \frac{2}{3} & \frac{2}{3}\varphi_{1,3}-\frac{4}{3}\varphi_{2,3} & \frac{4}{3}\varphi_{2,3} \\
      \hline & \varphi_1-\frac{9}{2}\varphi_2+9\varphi_3 & 6\varphi_2-18\varphi_3 & -\frac{3}{2}\varphi_2+9\varphi_3
    \end{array} 
\end{equation}

\subsubsection{Hochbruck--Ostermann ETDRK3}
Applying Algorithm~\ref{alg:htaumax} to \eqref{tab:HO_ETDRK3_1} produces
\begin{equation*}
  H(x)=A(x)\left(1+\frac{2x}{9c_2}\right)-1,
  \qquad x\in[0,1),
\end{equation*}
where
\[
  A(x)\coloneqq \frac{\sqrt{3}}{2}\frac{x}{\sqrt{1-x}}+\frac{3}{4}x.
\]
Then
\[
  \partial_{c_2}H(x;c_2)=-A(x)\frac{2x}{9c_2^2}<0,\qquad
  \partial_x H(x;c_2)=A'(x)\Bigl(1+\frac{2x}{9c_2}\Bigr)+A(x)\frac{2}{9c_2}.
\]
Since
\[
  A'(x)=\frac{\sqrt{3}}{4}\frac{2-x}{(1-x)^{3/2}}+\frac{3}{4}>0,\qquad x\in(0,1),
\]
it follows that $\partial_x H(x;c_2)>0$. 
For each $c_2\in(0,1]$, the smallest positive root $\hxzero(c_2)$ in \eqref{def_x0} satisfies $H(\hxzero(c_2);c_2)=0$, and implicit differentiation gives
\[
  \hxzero'(c_2)=-\frac{\partial_{c_2}H(\hxzero(c_2);c_2)}{\partial_x H(\hxzero(c_2);c_2)}>0.
\]
Hence $\hxzero(c_2)$ is strictly increasing in $c_2$.
The maximum over $(0,1]$ is attained at $c_2=1$, where $\hxzero(1)\approx0.4677$ and $\htaumax(1)\approx0.4677\,\frac{\varepsilon}{C_L}$.

The remaining third-order schemes are treated by Algorithm~\ref{alg:htaumax}.
Cox--Matthews \eqref{tab:CM_ETDRK3} yields \(\hxzero\approx0.3832\), hence \(\htaumax\approx0.3832\,\varepsilon/C_L\).
Celledoni--Marthinsen--Owren \eqref{tab:CMO_ETDRK3} yields \(\hxzero\approx0.2222\), hence \(\htaumax\approx0.2222\,\varepsilon/C_L\).
See also Table~\ref{tab:hxzero_classical}.

\subsection{Fourth-order methods} \label{subsec:fourth_order}

In this subsection we evaluate the step-size restriction $\htaumax$ from Theorem~\ref{main_thm} for four representative schemes: the stiff-order Hochbruck--Ostermann ETDRK4 method \cite{Hochbruck2005SINUM_OrderCondition}, and the classical-order Cox--Matthews \cite{Cox2002JCP_ETD}, Krogstad \cite{Krogstad2005JCP}, and Strehmel--Weiner \cite{BOOK_Strehmel1992} ETDRK4 methods.
We record their Butcher tableaux below.

The stiff-order Hochbruck--Ostermann ETDRK4 scheme is
\begin{equation}\label{tab:HO_ETDRK4}
  \renewcommand{\arraystretch}{1.2}
  \begin{array}{c|ccccc}
    0           & ~                                           & ~             & ~             & ~                                 & ~                         \\
    \frac{1}{2} & \frac{1}{2} \varphi_{1,2}                   & ~             & ~             & ~                                 & ~                         \\
    \frac{1}{2} & \frac{1}{2} \varphi_{1,3}-\varphi_{2,3}     & \varphi_{2,3} & ~             & ~                                 & ~                         \\
    1           & \varphi_{1,4}-2 \varphi_{2,4}               & \varphi_{2,4} & \varphi_{2,4} & ~                                 & ~                         \\
    \frac{1}{2} & \frac{1}{2} \varphi_{1,5}-2 a_{5,2}-a_{5,4} & a_{5,2}       & a_{5,2}       & \frac{1}{4} \varphi_{2,5}-a_{5,2} & ~                         \\
    \hline
    ~           & \varphi_1 - 3 \varphi_2 + 4 \varphi_3       & 0             & 0             & -\varphi_2 + 4 \varphi_3          & 4 \varphi_2 - 8 \varphi_3
  \end{array},
\end{equation}
with
\[
  a_{52} = \frac{1}{2} \varphi_{2,5} - \varphi_{3,4} + \frac{1}{4} \varphi_{2,4} -\frac{1}{2} \varphi_{3,5}.
\]
The three classical-order ETDRK4 schemes are the Cox--Matthews method
\begin{equation}\label{tab:CM_ETDRK4}
  \renewcommand{\arraystretch}{1.2}
  \begin{array}{c|cccc}
    0 & & & & \\
    \frac{1}{2} & \frac{1}{2}\varphi_{1,2} & & & \\
    \frac{1}{2} & 0 & \frac{1}{2}\varphi_{1,3} & & \\
    1 & \frac{1}{2}\varphi_{1,3}(\varphi_{0,3}-1) & 0 & \varphi_{1,3} & \\
    \hline
      & \varphi_1-3\varphi_2+4\varphi_3
      & 2\varphi_2-4\varphi_3
      & 2\varphi_2-4\varphi_3
      & 4\varphi_3-\varphi_2
  \end{array}.
\end{equation}
the Krogstad method
\begin{equation}\label{tab:Krogstad_ETDRK4}
  \renewcommand{\arraystretch}{1.2}
  \begin{array}{c|cccc}
    0 & & & & \\
    \frac{1}{2} & \frac{1}{2}\varphi_{1,2} & & & \\
    \frac{1}{2} & \frac{1}{2}\varphi_{1,3}-\varphi_{2,3} & \varphi_{2,3} & & \\
    1 & \varphi_{1,4}-2\varphi_{2,4} & 0 & 2\varphi_{2,4} & \\
    \hline
      & \varphi_1-3\varphi_2+4\varphi_3
      & 2\varphi_2-4\varphi_3
      & 2\varphi_2-4\varphi_3
      & -\varphi_2+4\varphi_3
  \end{array}.
\end{equation}
and the Strehmel--Weiner method
\begin{equation}\label{tab:SW_ETDRK4}
  \renewcommand{\arraystretch}{1.2}
  \begin{array}{c|cccc}
    0 & & & & \\
    \frac{1}{2} & \frac{1}{2}\varphi_{1,2} & & & \\
    \frac{1}{2} & \frac{1}{2}\varphi_{1,3}-\frac{1}{2}\varphi_{2,3} & \frac{1}{2}\varphi_{2,3} & & \\
    1 & \varphi_{1,4}-2\varphi_{2,4} & -2\varphi_{2,4} & 4\varphi_{2,4} & \\
    \hline
      & \varphi_1-3\varphi_2+4\varphi_3
      & 0
      & 4\varphi_2-8\varphi_3
      & -\varphi_2+4\varphi_3
  \end{array}.
\end{equation}

The four tableaux are treated by Algorithm~\ref{alg:htaumax}.
Hochbruck--Ostermann \eqref{tab:HO_ETDRK4} yields \(\hxzero\approx0.3675\).
Cox--Matthews \eqref{tab:CM_ETDRK4} and Krogstad \eqref{tab:Krogstad_ETDRK4} yield identical integral bounds \(\halp_{ij}\) and \(\hbet_i\), and therefore the same root \(\hxzero\approx0.3887\).
Strehmel--Weiner \eqref{tab:SW_ETDRK4} yields \(\hxzero\approx0.3451\).
In each case \(\htaumax=\hxzero\,\varepsilon/C_L\); see also Table~\ref{tab:hxzero_classical}.

\subsection{Fifth-order methods}\label{subsec:fifth_order}

Luan and Ostermann \cite{Luan2014HighOrderExpRK} proved that no ETDRK method with six or fewer stages can achieve fifth-order convergence for semilinear stiff problems.
They then constructed the first ETDRK method of fifth order, requiring eight stages.
In this subsection we evaluate the step-size restriction $\htaumax$ from Theorem~\ref{main_thm} for the Luan--Ostermann ETDRK5 method; its Butcher tableau is
\begin{equation}\label{tab:LO_ETDRK5}
        \renewcommand{\arraystretch}{1.2}
        \resizebox{\textwidth}{!}{$
        \begin{array}{l|cccccccc}
            0  & ~& ~ & ~&   ~ &   ~ &   ~ &   ~ &   ~ \\
            \frac{1}{2} & \frac{1}{2} \varphi_{1,2}      &  ~ &   ~ &   ~ &   ~ &   ~ &   ~ &   ~ \\
            \frac{1}{2} & \frac{1}{2} \varphi_{1,3}-\frac{1}{2}\varphi_{2,3}        & \frac{1}{2}\varphi_{2,3} &   ~ &   ~ &   ~ &   ~ &   ~ &   ~ \\
            \frac{1}{4} & \frac{1}{4} \varphi_{1,4} - \frac{1}{8} \varphi_{2,4} & 0 &       \frac{1}{8} \varphi_{2,4}   &   ~ &   ~ &   ~ &   ~ &   ~ \\
            \frac{1}{2} & \frac{1}{2} \varphi_{1,5} - \frac{3}{2} \varphi_{2,5} + 2 \varphi_{3,5} & 0 & -\frac{1}{2} \varphi_{2,5} + 2 \varphi_{3,5} & 2 \varphi_{2,5} - 4 \varphi_{3,5} & ~ & ~ & ~ & ~ \\
            \frac{1}{5} & \frac{1}{5}\varphi_{1,6}-\frac{1}{2}a_{64}-\frac{2}{25}\varphi_{2,6} & 0 & 0 & a_{64}&\frac{2}{25} \varphi_{2,6}  -\frac{1}{2}a_{64}& & & \\
            \frac{2}{3} & a_{71} & 0 & 0 & -\frac{125}{162}a_{64} & a_{75} & a_{76} & & \\
            1 & \varphi_1 - a_{85} - a_{86} - a_{87} & 0 & 0 & 0 & a_{85} & a_{86} & a_{87} & \\
            \hline
            ~& \varphi_1 - b_6-b_7-b_8   & 0   & 0 & 0 & 0   & b_6  & b_7 & b_8
        \end{array},
        $}
    \end{equation}
    where
    \begin{align*}
      & a_{64}=\frac{8}{25} \varphi_{2,6}-\frac{32}{125} \varphi_{3,6}, \qquad a_{75} = \frac{125}{1944} a_{64}-\frac{16}{27} \varphi_{2,7}+\frac{320}{81} \varphi_{3,7}, \\
      & a_{76}=\frac{3125}{3888} a_{64}+\frac{100}{27} \varphi_{2,7}-\frac{800}{81} \varphi_{3,7}, \qquad a_{71} = \frac{2}{3}\varphi_{1,7}+\frac{125}{162}a_{64}-a_{75}-a_{76},\\
      & \phi=\frac{5}{32} a_{64}-\frac{1}{28} \varphi_{2,6}+\frac{36}{175} \varphi_{2,7}-\frac{48}{25} \varphi_{3,7}+\frac{6}{175} \varphi_{4,6}+\frac{192}{35} \varphi_{4,7}+6 \varphi_{4,8}, \\
      & a_{85}=\frac{208}{3} \varphi_{3,8}-\frac{16}{3} \varphi_{2,8}-40 \phi, \qquad a_{86}=-\frac{250}{3} \varphi_{3,8}+\frac{250}{21} \varphi_{2,8}+\frac{250}{7} \phi, \\
      & a_{87}=-27 \varphi_{3,8}+\frac{27}{14} \varphi_{2,8}+\frac{135}{7} \phi, \qquad b_6=\frac{125}{14} \varphi_2-\frac{625}{14} \varphi_3+\frac{1125}{14} \varphi_4, \\
      & b_7=-\frac{27}{14} \varphi_2+\frac{162}{7} \varphi_3-\frac{405}{7} \varphi_4, \qquad b_8=\frac{1}{2} \varphi_2-\frac{13}{2} \varphi_3+\frac{45}{2} \varphi_4 .
      \end{align*}

The Luan--Ostermann scheme \eqref{tab:LO_ETDRK5} is treated by Algorithm~\ref{alg:htaumax} and yields \(\hxzero\approx0.2651\), so \(\htaumax\approx0.2651\,\varepsilon/C_L\); see also Table~\ref{tab:hxzero_classical}.

\begin{table}[ht]
  \centering
  \caption{Smallest positive roots \(\hxzero\) in \eqref{def_x0} for second- to fifth-order ETDRK schemes.
  By Theorem~\ref{main_thm}, the energy-dissipation step-size restriction is
  \(\htaumax=\varepsilon \hxzero/C_L\).}
  \label{tab:hxzero_classical}
  \begin{tabular}{|c|l|c|}
    \hline
    Order & Scheme & \(\hxzero\) \\
    \hline
    2
      & Hochbruck--Ostermann \eqref{tab:HO_ETDRK2} ($c_2=1$)
      & $\approx 0.6667$ \\
    \hline
    \multirow{3}{*}{3}
      & Hochbruck--Ostermann \eqref{tab:HO_ETDRK3_1} ($c_2=1$)
      & $\approx 0.4677$ \\[0.35em]
    & Cox--Matthews \eqref{tab:CM_ETDRK3}
      & $\approx 0.3832$ \\[0.35em]
    & Celledoni--Marthinsen--Owren \eqref{tab:CMO_ETDRK3}
      & $\approx 0.2222$ \\
    \hline
    \multirow{4}{*}{4}
      & Hochbruck--Ostermann \eqref{tab:HO_ETDRK4}
      & $\approx 0.3675$ \\[0.35em]
    & Cox--Matthews \eqref{tab:CM_ETDRK4}
      & $\approx 0.3887$ \\[0.35em]
    & Krogstad \eqref{tab:Krogstad_ETDRK4}
      & $\approx 0.3887$ \\[0.35em]
    & Strehmel--Weiner \eqref{tab:SW_ETDRK4}
      & $\approx 0.3451$ \\
    \hline
    5
      & Luan--Ostermann \eqref{tab:LO_ETDRK5}
      & $\approx 0.2651$ \\
    \hline
  \end{tabular}
\end{table}

\subsection{Arbitrarily high-order interpolation-based ETDRK schemes}\label{subsec:interpolation_ETDRK}

In \cite{Quan_25_IMA_ETD_MP}, a family of interpolation-based ETDRK methods of arbitrary order was constructed for $L^2$ gradient flows, and the original energy was shown to decrease under an explicit step-size restriction.
That restriction is rather severe: the admissible step size decreases exponentially as the order increases.
We first recall the construction, then apply Theorem~\ref{main_thm} to the same family and obtain a substantially milder bound.

Let $m\ge 2$ denote the order of the scheme.
The method is built recursively from the variation-of-constants formula by interpolating the nonlinear term on the current step $[0,\tau]$.
For each level $k\ge 1$, one chooses $k+1$ interpolation nodes
$0=\theta_{k,0}<\cdots<\theta_{k,k}=1$ on $[0,1]$, with Lagrange basis functions
\begin{equation*}
    \ell_{k,j}(\theta)=\prod_{\substack{\ell=0\\ \ell\neq j}}^{k}\frac{\theta-\theta_{k,\ell}}{\theta_{k,j}-\theta_{k,\ell}},\qquad \theta\in[0,1].
\end{equation*}
For $k=0$ we write $\ell_{0,0}\equiv 1$.
Starting from $P_0(\sigma)\equiv g(u^n)$, the approximations $w^{[k]}$ are defined recursively by
\begin{equation*}
    w^{[k+1]}(s)=\e^{s\,\varepsilon\hL}u^n+\int_0^s \e{(s-\sigma)\varepsilon\hL}P_k(\sigma)\,\dd\sigma,
    \qquad k=0,1,\dots,m-1,
\end{equation*}
and, for $k=1,\dots,m-1$,
\begin{equation*}
    P_k(\sigma)=\sum_{j=0}^{k}\ell_{k,j}\!\left(\frac{\sigma}{\tau}\right)\,g\bigl(w^{[k]}(\theta_{k,j}\tau)\bigr),\qquad \sigma\in[0,\tau].
\end{equation*}
The update is $u^{n+1}=w^{[m]}(\tau)$.
Since $\theta_{k,0}=0$, one has $w^{[k]}(0)=u^n$ for every $k$, so that node reuses the first stage $u^n$ and is not recomputed.
The resulting scheme is an $m$th-order ETDRK method \cite{Quan_25_IMA_ETD_MP}, and it preserves equilibria because $\sum_{j=0}^{k}\ell_{k,j}\equiv 1$.

Written in the form \eqref{eq:ETDRK}, the stages consist of $u^n$ together with
$w^{[k]}(\theta_{k,j}\tau)$ for $k=1,\ldots,m-1$ and $j=1,\ldots,k$.
Hence the number of stages is
\begin{equation*}
  r=1+\sum_{k=1}^{m-1}k=1+\frac{m(m-1)}{2}.
\end{equation*}
The abscissa of the stage $w^{[k]}(\theta_{k,j}\tau)$ is $c_i=\theta_{k,j}$.
After the substitution $\sigma=\theta\tau$, the internal and weight coefficients admit the integral representations
\begin{equation}\label{eq:interp_aij}
    a_{ij}(\tau\varepsilon\hL)=\int_0^{c_i}\e{(c_i-\theta)\tau\varepsilon\hL}p_{ij}(\theta)\,\dd\theta,
    \qquad 0 \leq c_i\leq 1,
\end{equation}
and
\begin{equation*}
    b_j(\tau\varepsilon\hL)=\int_0^1\e^{(1-\theta)\tau\varepsilon\hL}q_j(\theta)\,\dd\theta,
\end{equation*}
where $p_{ij}(\theta)$ and $q_j(\theta)$ are the Lagrange polynomials associated with the construction above, as recorded in
\begin{equation*}
\renewcommand{\arraystretch}{1.2}
\resizebox{\dimexpr\displaywidth-3.5em\relax}{!}{$
\begin{array}{c|c}
    \left(c_{i}\right)_{r \times 1} & \left(p_{i, j}\right)_{r \times r}  \\
    \hline
     & \left(q_{j}\right)_{1 \times r}
\end{array}=
\begin{array}{c|ccccccccc}
    \theta_{1,0} & & & & & & \\
    \theta_{1,1} & \ell_{0,0} & & & & & \\
    \theta_{2,1} & \ell_{1,0} & \ell_{1,1} & & & & \\
    \theta_{2,2} & \ell_{1,0} & \ell_{1,1} & 0 & & & \\
    \vdots & & & & & & \\
    \theta_{m-1,1} & \ell_{m-2,0} & 0 & \cdots & 0 & \ell_{m-2,1} & \cdots & \ell_{m-2,m-2} \\
    \vdots & & & & & & \\
    \theta_{m-1,m-1} & \ell_{m-2,0} & 0 & \cdots & 0 & \ell_{m-2,1} & \cdots & \ell_{m-2,m-2} \\
    \hline
    ~& \ell_{m-1,0} & 0 & \cdots & 0 & 0 & \ell_{m-1,1} & \cdots & \ell_{m-1,m-1}
\end{array}.
$}
\end{equation*}
Stages at the same level share the interpolant of the previous level and differ only in the abscissa (the upper limit of integration).
Since $\vertiii{\e{(c_i-\theta)\tau\varepsilon\hL}}\le1$ and the integral in \eqref{eq:interp_aij} only runs up to $c_i$, we take
\begin{equation*}
    \halp_{ij}=\int_0^{c_i}|p_{ij}(\theta)|\,\dd\theta,
    \qquad 
    \hbet_j=\int_0^1|q_j(\theta)|\,\dd\theta.
\end{equation*}
These bounds play the same role as the $\varphi$-function expansion in Algorithm~\ref{alg:htaumax}.
Applying Theorem~\ref{main_thm} to this family of ETDRK methods gives the step-size restriction 
\begin{equation}\label{eq:Mrnew}
\tau\leq \widehat{\tau}_{\max,m}^{\mathrm{new}}=\frac{\varepsilon}{C_L}M_m^{\text{new}},
\end{equation}
where $M_m^{\text{new}}\coloneqq \hxzero$ is the smallest positive root of $H$ in \eqref{def_H} evaluated with \(\halp_{ij}\) and \(\hbet_i\).

In \cite{Quan_25_IMA_ETD_MP}, the same interpolation-based scheme is proved to preserve the original energy dissipation under the step-size restriction 
\begin{equation}\label{eq:Mrold}
    \tau\le\widehat{\tau}_{\max,m}^{\mathrm{old}}=\frac{\varepsilon}{C_L}M_m^{\text{old}}.
\end{equation} 
Table~\ref{tab:tau_max_comparison_Quan} reports $M_m^{\text{new}}$ and $M_m^{\text{old}}$ for schemes of order $m=2,\ldots,8$, using either equispaced or Chebyshev interpolation nodes.
It can be observed that the new bounds decrease much more slowly with~$m$, whereas the old bounds decrease exponentially with~$m$.
For Chebyshev nodes at order $m=8$, $M_8^{\text{new}}\approx0.2387$ compared with $M_8^{\text{old}}\approx1.7\times10^{-6}$, so the admissible time step is larger by a factor of about $1.4\times10^{5}$.
Even on equispaced nodes, where $M_m^{\text{new}}$ decreases faster than on Chebyshev nodes, the improvement over \cite{Quan_25_IMA_ETD_MP} is still dramatic ($M_8^{\text{new}}\approx0.1512$ versus $M_8^{\text{old}}\approx5.3\times10^{-7}$).

\begin{table}[htbp]
  \centering
  \caption{$M_m^{\text{new}}$ in \eqref{eq:Mrnew} from Theorem~\ref{main_thm} and $M_m^{\text{old}}$ in \eqref{eq:Mrold} from \cite{Quan_25_IMA_ETD_MP}, for schemes of order $m$ with Chebyshev or equispaced interpolation nodes.}
  \label{tab:tau_max_comparison_Quan}
  \begin{tabular}{|c|c|c|c|c|}
    \hline
    \multirow{2}{*}{Order $m$} & \multicolumn{2}{c|}{Chebyshev nodes} & \multicolumn{2}{c|}{Equispaced nodes} \\
    \cline{2-5}
     & $M_m^{\text{new}}$ & $M_m^{\text{old}}$ & $M_m^{\text{new}}$ & $M_m^{\text{old}}$ \\
    \hline
    2 & 0.6667 & 0.250 & 0.6667 & 0.250 \\
    3 & 0.4196 & 0.021 & 0.4196 & 0.021 \\
    4 & 0.3296 & $2.8{\times}10^{-3}$ & 0.3256 & $2.3{\times}10^{-3}$ \\
    5 & 0.2890 & $4.3{\times}10^{-4}$ & 0.2712 & $2.8{\times}10^{-4}$ \\
    6 & 0.2657 & $6.7{\times}10^{-5}$ & 0.2282 & $3.4{\times}10^{-5}$ \\
    7 & 0.2501 & $1.1{\times}10^{-5}$ & 0.1888 & $4.2{\times}10^{-6}$ \\
    8 & 0.2387 & $1.7{\times}10^{-6}$ & 0.1512 & $5.3{\times}10^{-7}$ \\
    \hline
  \end{tabular}
\end{table}

\section{Numerical experiments}\label{sec:numerical_experiments}

In this section, we verify the energy-dissipation step-size restriction of Theorem~\ref{main_thm} on the Allen--Cahn equation.
All experiments use the ETDRK schemes \eqref{eq:ETDRK} applied directly to the gradient flow, without a stabilization term, in contrast to the energy analysis for stabilized ETDRK methods \cite{Du2021SIAMRev_MP4ETD,Fu2022JCP_ED4ETD2}.
We consider
\begin{equation*}
    \partial_t u = \varepsilon\,\Delta u + \frac{1}{\varepsilon}\bigl(u-u^3\bigr)
\end{equation*}
on the domain $\Omega=(0,1)$ with $\varepsilon=0.02$, periodic boundary conditions and initial condition \(u^0(x)=0.1\sin(2\pi x)\).
We simulate until final time $T=1$.
The spatial discretization uses the second-order central-difference Laplacian $\Delta_h$ on a uniform mesh with $N=1024$ grid points.
Since $\Delta_h$ is self-adjoint and negative semidefinite on the periodic grid, $\{e^{t\varepsilon\Delta_h}\}$ is a contraction semigroup, and the energy-dissipation result of Theorem~\ref{main_thm} still applies to the fully discrete scheme by Remark~\ref{rem:fully_discrete}. In this setting, the product of the exponential of discrete Laplace matrix and a vector can be accelerated with the fast Fourier transformation.
The discrete energy associated with \eqref{eq:E} is
\begin{equation*}
  E_h(u)=-\frac{\varepsilon}{2}(\Delta_h u,u)_h+\frac{h}{\varepsilon}\sum_j F(u_j),\qquad F(u)=\frac14(u^2-1)^2,
\end{equation*}
where $h=1/N$ is the mesh size and $(\cdot,\cdot)_h$ is the discrete $L^2$ inner product.
We write $E^n=E_h(u^n)$.
The functions $\varphi_k$ are computed using the contour-averaging method with $M=64$ quadrature nodes and contour radius $R=1$, as suggested in \cite{KassamTrefethen2005SISC_ETDRK4}.
In the proof of Theorem~\ref{main_thm} we assume that $f$ is globally Lipschitz continuous with constant $C_L$.
For the Allen–Cahn nonlinearity, we consider a globally Lipschitz extension \(\widetilde f\) of \(f(u)=u-u^3\) outside \([-1.15,1.15]\), chosen so that \(\widetilde f=f\) on this interval and the derived Lipschitz constant \(\mathrm{Lip}(\widetilde f)\le3\) (see the construction of $\widetilde f$ in \cite{ShenYang2010DCDS_ACandCH}). Since all computed stage values and numerical solutions remain in this interval for \(\tau\le\widehat\tau_{\max}\), the numerical ETDRK iterates coincide with those obtained from the original cubic nonlinearity.

We test the following six schemes: Hochbruck--Ostermann ETDRK2 \eqref{tab:HO_ETDRK2} with \(c_2=1\), Hochbruck--Ostermann ETDRK3 \eqref{tab:HO_ETDRK3_1} with \(c_2=1\), Krogstad ETDRK4 \eqref{tab:Krogstad_ETDRK4}, Luan--Ostermann ETDRK5 \eqref{tab:LO_ETDRK5}, interpolation-based ETDRK6, and interpolation-based ETDRK8, both on Chebyshev nodes from Section~\ref{subsec:interpolation_ETDRK}.
For each scheme, \(\tau\) runs through logarithmically spaced multiples of \(\htaumax\).
We write \(\ttaumax\) for the largest of these values at which \(E^{n+1}\le E^n\) at every step.
Table~\ref{tab:exp_verify} records \(\htaumax\) and \(\ttaumax\), together with the earlier result on original energy dissipation for each scheme.
In this table, \(E^{n+1}\le E^n\) holds at every step whenever \(\tau\le\htaumax\), in agreement with Theorem~\ref{main_thm}, while \(\ttaumax\) is several times \(\htaumax\), so the sufficient condition is not sharp.
For Hochbruck--Ostermann ETDRK2 and ETDRK3, earlier work gave unconditional original energy dissipation only after a stabilization term is added \cite{Fu2022JCP_ED4ETD2,Fu2025ChinaMath_ED4HighOrderETD}, and the unstabilized case was not known.
For Krogstad ETDRK4 and Luan--Ostermann ETDRK5, Theorem~\ref{main_thm} appears to be the first original energy dissipation result.
For interpolation-based ETDRK6 and ETDRK8, the experiments in \cite{Quan_25_IMA_ETD_MP} already indicated that the numerical threshold is far larger than \(\widehat{\tau}_{\max,m}^{\mathrm{old}}\), and the new bound \(\widehat{\tau}_{\max,m}^{\mathrm{new}}\) now lies only a moderate factor below \(\ttaumax\).

\begin{table}[ht]
  \centering
  \caption{Theoretical step size \(\htaumax\) and numerical threshold \(\ttaumax\) (\(\varepsilon=0.02\), \(N=1024\), \(T=1\)).
  Hochbruck--Ostermann ETDRK2 and ETDRK3 use \(c_2=1\).
  Interpolation-based ETDRK6 and ETDRK8 use Chebyshev nodes.
  A dash means that no earlier energy result is known for that scheme.}
  \label{tab:exp_verify}
  \begin{tabular}{|c|c|c|c|}
    \hline
    Scheme & Earlier energy result & \(\htaumax\) & \(\ttaumax\) \\
    \hline
    \begin{tabular}{@{}c@{}}Hochbruck--Ostermann\\ ETDRK2 \eqref{tab:HO_ETDRK2}\end{tabular}
      & \begin{tabular}{@{}c@{}}unconditional dissipation,\\ with stabilization \cite{Fu2022JCP_ED4ETD2}\end{tabular}
      & \(0.00444\) & \(0.0229\) 
      \\
    \hline
    \begin{tabular}{@{}c@{}}Hochbruck--Ostermann\\ ETDRK3 \eqref{tab:HO_ETDRK3_1}\end{tabular}
      & \begin{tabular}{@{}c@{}}unconditional dissipation,\\ with stabilization \cite{Fu2025ChinaMath_ED4HighOrderETD}\end{tabular}
      & \(0.00312\) & \(0.0234\) 
      \\
    \hline
    \begin{tabular}{@{}c@{}}Krogstad\\ ETDRK4 \eqref{tab:Krogstad_ETDRK4}\end{tabular}
      & ---
      & \(0.00259\) & \(0.0284\) 
      \\
    \hline
    \begin{tabular}{@{}c@{}}Luan--Ostermann\\ ETDRK5 \eqref{tab:LO_ETDRK5}\end{tabular}
      & ---
      & \(0.00177\) & \(0.0378\) 
      \\
    \hline
    \begin{tabular}{@{}c@{}}Interpolation-based\\ ETDRK6 \eqref{eq:interp_aij}\end{tabular}
      & \(\widehat{\tau}_{\max,6}^{\mathrm{old}}=4.5\times 10^{-7}\) \cite{Quan_25_IMA_ETD_MP}
      & \(0.00177\) & \(0.0329\) 
      \\
    \hline
    \begin{tabular}{@{}c@{}}Interpolation-based\\ ETDRK8 \eqref{eq:interp_aij}\end{tabular}
      & \(\widehat{\tau}_{\max,8}^{\mathrm{old}}=1.1\times 10^{-8}\) \cite{Quan_25_IMA_ETD_MP}
      & \(0.00159\) & \(0.0370\) 
      \\
    \hline
  \end{tabular}
\end{table}

We take interpolation-based ETDRK8 as an example.
The old restriction \(\widehat{\tau}_{\max,8}^{\mathrm{old}}\) is about \(10^{-8}\) and cannot be used in practice, while \(\widehat{\tau}_{\max,8}^{\mathrm{new}}\) is about \(10^{-3}\) and and can be used in practice.
Figure~\ref{fig:exp_interp_m8_energy} records \(E^n\) (top) and \(\|u^n\|_\infty\) (bottom) at \(\tau=10^{4}\,\widehat{\tau}_{\max,8}^{\mathrm{old}}\), \(\widehat{\tau}_{\max,8}^{\mathrm{new}}\), and \(30\,\widehat{\tau}_{\max,8}^{\mathrm{new}}\).
At the first two step sizes, \(E^{n+1}\le E^n\) at every step and \(\|u^n\|_\infty\le 1\).
At \(\tau=30\,\widehat{\tau}_{\max,8}^{\mathrm{new}}\), the discrete energy oscillates.

For Hochbruck--Ostermann ETDRK2 \eqref{tab:HO_ETDRK2} with \(c_2=1\), we check whether \(\ttaumax\) depends on the spatial mesh and whether it scales with \(\varepsilon\).
The experiments in Table~\ref{tab:exp_verify} were repeated at \(\varepsilon=0.02\) on the meshes \(N\in\{256,512,1024,2048\}\).
The value of \(\ttaumax\) is the same at every \(N\), so the numerical threshold is independent of the spatial mesh.
The same experiments were repeated at \(N=1024\) and \(\varepsilon\in\{0.01,0.02,0.04\}\).
At each of these three values, \(\ttaumax\) is about \(5.16\,\htaumax\), so the numerical threshold scales linearly with \(\varepsilon\), in agreement with Theorem~\ref{main_thm}.

\begin{figure}[ht]
  \centering
  \includegraphics[width=\textwidth]{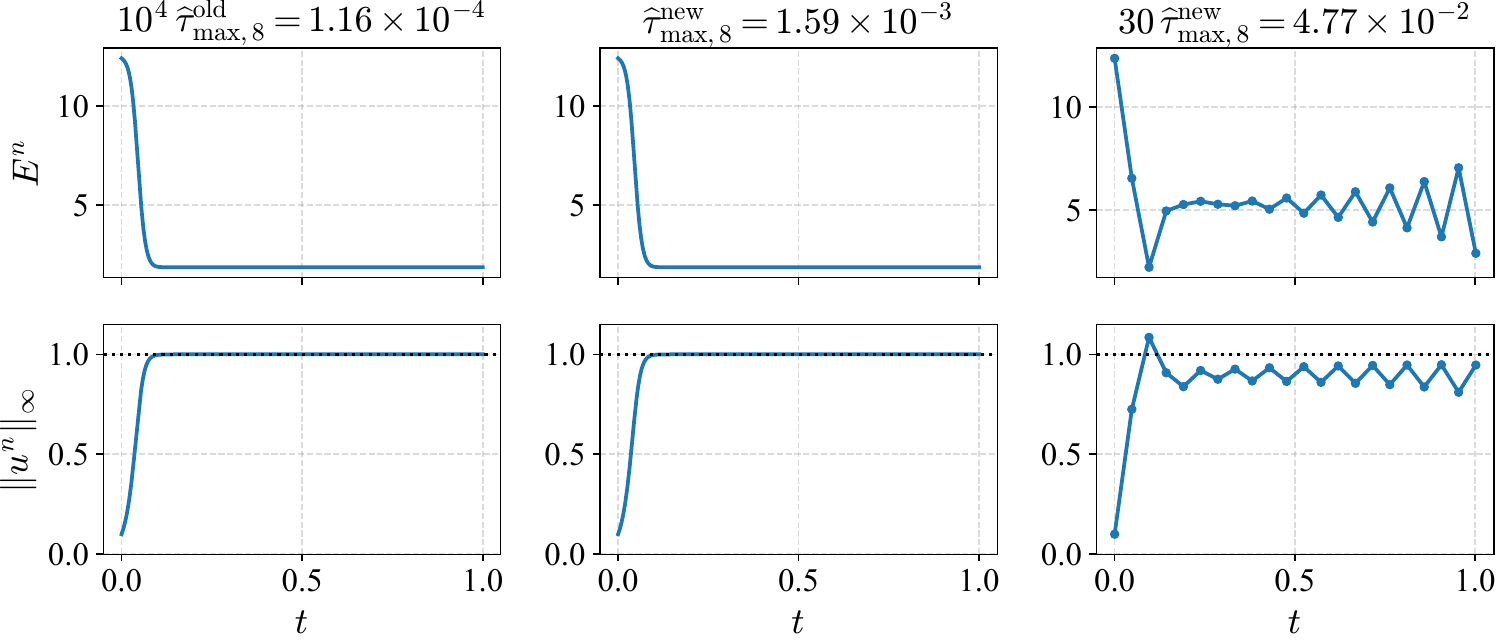}
  \caption{Discrete energy \(E^n\) (top) and infinity norm \(\|u^n\|_\infty\) (bottom) for interpolation-based ETDRK8 on Chebyshev nodes, at \(\tau=10^{4}\,\widehat{\tau}_{\max,8}^{\mathrm{old}}\), \(\widehat{\tau}_{\max,8}^{\mathrm{new}}\), and \(30\,\widehat{\tau}_{\max,8}^{\mathrm{new}}\) (\(\varepsilon=0.02\), \(N=1024\), \(T=1\)). The dashed line is \(1\).}
  \label{fig:exp_interp_m8_energy}
\end{figure}

\section{Conclusion}\label{sec:conclusion}

In this work, we have established a unified framework for the energy dissipation analysis of ETDRK schemes applied to $L^2$ gradient flows, without imposing positive-definiteness constraints on the Butcher tableau coefficients or adding stabilization terms.
We prove that for a broad class of equilibrium-preserving ETDRK methods, the original energy decreases under a mild step-size restriction, and we derive explicit restrictions for several classical ETDRK schemes.
Extending this analysis to non-\(L^2\) gradient flows and discovering high-order ETDRK schemes with unconditional original energy dissipation are natural directions for future work.

\bibliographystyle{siamplain}
\bibliography{refs}

\end{document}